\documentclass[12pt,leqno]{amsart}\usepackage[latin9]{inputenc}\usepackage{verbatim,amsthm,amstext,amssymb,color,amscd,bm}\makeatletter
\usepackage[colorlinks=true]{hyperref}\usepackage[shortlabels]{enumitem}\hypersetup{urlcolor=red,citecolor=blue}\numberwithin{equation}{section}
\theoremstyle{plain}\newtheorem{theorem}{Theorem}[section]
\newtheorem{lemma}[theorem]{Lemma}\newtheorem{corollary}[theorem]{Corollary}\theoremstyle{definition}\newtheorem{remark}[theorem]{Remark}
\DeclareMathOperator{\Widim}{\mathrm{Widim}}\DeclareMathOperator{\mdim}{\mathrm{mdim}}\DeclareMathOperator{\supp}{\mathrm{supp}}
\makeatother
\begin{document}
\title[Mean dimension of Bernstein spaces and universal real flows]{Mean dimension of Bernstein spaces\\and universal real flows}
\author[L. Jin]{Lei Jin}
\address{Lei Jin: Center for Mathematical Modeling, University of Chile and UMI 2807 - CNRS}
\email{jinleim@impan.pl}
\author[Y. Qiao]{Yixiao Qiao}
\address{Yixiao Qiao: School of Mathematical Sciences, South China Normal University, Guangzhou, Guangdong 510631, China}
\email{yxqiao@mail.ustc.edu.cn}
\author[S. Tu]{Siming Tu}
\address{Siming Tu: School of Mathematics (Zhuhai), Sun Yat-sen University, Zhuhai, Guangdong 519082, China}
\email{tusiming3@mail.sysu.edu.cn}
\begin{abstract}
We study the action of translation on the spaces of uniformly bounded continuous functions on the real line which are uniformly band-limited in a compact interval. We prove that two intervals themselves will decide if two spaces are topologically conjugate, while the length of an interval tells the mean dimension of a space. We also investigate universal real flows. We construct a sequence of compact invariant subsets of a space consisting of uniformly bounded smooth one-Lipschitz functions on the real line, which have mean dimension equal to one, such that all real flows can be equivariantly embedded in the translation on their product space. Moreover, we show that the countable self-product of any among them does not satisfy such a universal property. This, on the one hand, presents a more reasonable choice of a universal real flow with a view towards mean dimension, and on the other hand, clarifies a seemingly plausible impression; meanwhile, it refines the previously known results in this direction. Our proof goes through an approach of harmonic analysis. Furthermore, both the universal space that we provide and an embedding mapping which we build for any real flow are explicit.
\end{abstract}
\maketitle

\medskip

\section{Main results}
This paper is devoted to a study of the translation action on \textit{Bernstein spaces} and an alternative \textit{universal real flow} with a view towards \textit{mean dimension theory}. By a \textbf{real flow} (or an $\mathbb{R}$-action) we understand a triple $(X,\mathbb{R},T)$, where $X$ is a compact metric space and $$T:\mathbb{R}\times X\to X,\quad(t,x)\mapsto T_tx$$ is a continuous mapping satisfying: $$T_0x=x,\quad T_{r+s}x=T_rT_sx,\quad\forall x\in X,\;\forall r,s\in\mathbb{R}.$$ For two real flows $(X,\mathbb{R},T)$ and $(Y,\mathbb{R},S)$ we say that $(Y,\mathbb{R},S)$ can be \textbf{embedded} in $(X,\mathbb{R},T)$ if there is an equivariant topological embedding $\phi:Y\to X$, namely a homeomorphism $\phi$ of $Y$ into $X$ satisfying $$\phi(S_ty)=T_t\phi(y),\quad\forall t\in\mathbb{R},\;\forall y\in Y;$$ if in addition, such an equivariant mapping $\phi$ is a homeomorphism of $Y$ onto $X$, then $(X,\mathbb{R},T)$ and $(Y,\mathbb{R},S)$ are said to be \textbf{topologically conjugate} (to which some researchers may prefer the terminology ``isomorphic''). A real flow $(X,\mathbb{R},T)$ is called \textbf{universal} if all real flows can be embedded in $(X,\mathbb{R},T)$.

These terminologies for $\mathbb{R}$-actions may be applied similarly to \textit{discrete flows}, i.e. $\mathbb{Z}$-actions. A standard universal $\mathbb{Z}$-action is the \textit{Hilbert cube} $([-1,1]^\mathbb{N})^\mathbb{Z}$ under the $\mathbb{Z}$-translation: $$(x_k)_{k\in\mathbb{Z}}\longmapsto(x_{k+1})_{k\in\mathbb{Z}},\quad\forall x_k\in[-1,1]^\mathbb{N}.$$ Note that the Hilbert cube $([-1,1]^\mathbb{N})^\mathbb{Z}$ is a compact metric space. For $\mathbb{R}$-actions, an analogue of the Hilbert cube $([-1,1]^\mathbb{N})^\mathbb{Z}=([-1,1]^\mathbb{Z})^\mathbb{N}$ is the function space $C(\mathbb{R})^\mathbb{N}$, where $C(\mathbb{R})$ denotes the space of continuous functions $f:\mathbb{R}\to[-1,1]$ endowed with the topology of uniform convergence on compact subsets of $\mathbb{R}$, given by the distance:
\begin{equation}\label{distance}
D(f,g)=\sum_{n=1}^\infty\frac{|\!|f-g|\!|_{L^\infty([-n,n])}}{2^n}\quad\quad(f,g\in C(\mathbb{R})).
\end{equation}
Let the group $\mathbb{R}$ act on $C(\mathbb{R})$ continuously by the translation
\begin{equation}\label{translation}
\sigma:\mathbb{R}\times C(\mathbb{R})\to C(\mathbb{R}),\quad\quad(t,f(\cdot))\mapsto f(\cdot+t).
\end{equation}
In the same way as in $\mathbb{Z}$-actions, we can embed all real flows in the translation on the product space $C(\mathbb{R})^\mathbb{N}$ naturally. Unfortunately, if we try to consider the translation on $C(\mathbb{R})^\mathbb{N}$ as a ``universal real flow'' then there is a problem: The space $C(\mathbb{R})^\mathbb{N}$ is not compact, nor locally compact. So it is not a ``real flow'' in the definition.

\medskip

We expect to find a universal real flow as simple as possible. Nevertheless, it would be less interesting if a universal space is ``larger'' than the function space $C(\mathbb{R})^\mathbb{N}$. This poses the following question:
\begin{itemize}\item
Is there an ``explicit'' \textit{compact} invariant subset of $C(\mathbb{R})^\mathbb{N}$ that is universal?
\end{itemize}
Here ``explicitness'' means that we may characterize all elements in a chosen space easily. Answering the above question positively, Gutman and Jin \cite{GJ} successfully constructed a countable product of compact invariant subsets of $C(\mathbb{R})$, which is universal under the translation.

\medskip

To state this result in a precise way, we briefly recall some necessary notions and results in Fourier analysis. A \textbf{rapidly decreasing function} is an infinitely differentiable function $f$ on $\mathbb{R}$ satisfying $$\lim_{|t|\to+\infty}t^nf^{(j)}(t)=0,\quad\forall n,j\in\mathbb{N}.$$ A \textbf{tempered distribution} on $\mathbb{R}$ is a continuous linear functional on the space of all rapidly decreasing functions equipped with the topology given by a family of seminorms as follows: $$|\!|f|\!|_{j,n}=\sup_{t\in\mathbb{R}}|t^nf^{(j)}(t)|\quad(j,n\in\mathbb{N}).$$ The tempered distributions include in particular bounded continuous functions. For rapidly decreasing functions $f:\mathbb{R}\to\mathbb{C}$ the definition of the \textbf{Fourier transforms} of $f$ is given by $$\mathcal{F}(f)(\xi)=\int_{-\infty}^{+\infty}e^{-2\pi\sqrt{-1}t\xi}f(t)dt,\quad\overline{\mathcal{F}}(f)(t)=\int_{-\infty}^{+\infty}e^{2\pi\sqrt{-1}t\xi}f(\xi)d\xi.$$ The operators $\mathcal{F}$ and $\overline{\mathcal{F}}$ can be extended to tempered distributions in a standard way (for details we refer to \cite[Chapter 7]{Schwartz}).

\medskip

Let $I$ be a compact subset of $\mathbb{R}$. A bounded continuous function $f:\mathbb{R}\to\mathbb{C}$ is \textbf{band-limited} in $I$ if $\supp(\mathcal{F}(f))\subset I$, meaning that $\langle\mathcal{F}(f),g\rangle=0$ for all rapidly decreasing functions $g:\mathbb{R}\to\mathbb{C}$ with $\supp(g)\cap I=\emptyset$. We denote by $\mathcal{B}^\mathbb{C}(I)$ (resp. $\mathcal{B}(I)$) the set of continuous functions $f:\mathbb{R}\to\mathbb{C}$ (resp. $f:\mathbb{R}\to\mathbb{R}$) band-limited in $I$ with $|\!|f|\!|_{L^\infty(\mathbb{R})}\le1$. Clearly, both $\mathcal{B}^\mathbb{C}(I)$ and $\mathcal{B}(I)$ are invariant under the translation $\sigma$ defined in \eqref{translation}. An important and nontrivial fact \cite[Lemma 2.3]{GT}\cite[Chapter 7, Section 4]{Schwartz} is that if $I\subset\mathbb{R}$ is compact then $\mathcal{B}^\mathbb{C}(I)$ and $\mathcal{B}(I)$ are compact metric spaces with respect to the distance $D$ given in \eqref{distance} which coincides with the standard topology of tempered distributions. Thus, $(\mathcal{B}^\mathbb{C}(I),\mathbb{R},\sigma)$ and $(\mathcal{B}(I),\mathbb{R},\sigma)$ become real flows.

\begin{remark}
The distance $D$ defined in \eqref{distance} and the translation $\sigma$ defined in \eqref{translation} should be understood a distance and an action of $\mathbb{R}$, respectively, on the space of continuous functions $f:\mathbb{R}\to\mathbb{C}$ (which is larger than $C(\mathbb{R})$). We do not change the notation here as it does not cause any confusion.
\end{remark}

\medskip

\begin{theorem}[{\cite[Theorem 1.2]{GJ}}]\label{gj}
Under the translation $\sigma$ the space $\prod_{n\in\mathbb{N}}\mathcal{B}([-n,n])$ is universal.
\end{theorem}

Although Theorem \ref{gj} provides an affirmative solution to the above question, it is not so satisfactory because in contrast to the Hilbert cube $([-1,1]^\mathbb{Z})^\mathbb{N}$, the universal space $\prod_{n\in\mathbb{N}}\mathcal{B}([-n,n])$ appearing in Theorem \ref{gj} is not a \textit{self-product}. To proceed, we would like to seek a universal space which is ``closer'' to the Hilbert cube $([-1,1]^\mathbb{Z})^\mathbb{N}$. Hence a natural question arises as follows:
\begin{itemize}\item
Is there an explicit \textit{compact} invariant subset $F$ of $C(\mathbb{R})$ such that $F^\mathbb{N}$ is universal?
\end{itemize}
Under this motivation, Jin and Tu \cite{JT} found that the one-Lipschitz function space is a solution.

\medskip

Formally, we let $L(\mathbb{R})$ be the set of all functions $f:\mathbb{R}\to[-1,1]$ with the following property: $$|f(s)-f(t)|\le|s-t|,\quad\forall s,t\in\mathbb{R}.$$ Obviously, $L(\mathbb{R})$ is an invariant subset of $C(\mathbb{R})$ under the translation $\sigma$ in \eqref{translation}. Moreover, by the Arzela--Ascoli theorem, $L(\mathbb{R})$ is a compact metric space with respect to the distance $D$ in \eqref{distance}.

\begin{theorem}[{\cite[Theorem 1.2]{JT}}]\label{jt}
Under the translation $\sigma$ the space $L(\mathbb{R})^\mathbb{N}$ is universal.
\end{theorem}

\medskip

However, we are still not satisfied with Theorem \ref{jt} in spite of the fact that $L(\mathbb{R})^\mathbb{N}$ is a countable self-product of $L(\mathbb{R})\subset C(\mathbb{R})$, as a deeper observation reveals a more serious problem: The mean dimension of $L(\mathbb{R})$ (under the translation $\sigma$) is $+\infty$, while in the Hilbert cube $([-1,1]^\mathbb{Z})^\mathbb{N}$ the mean dimension of $[-1,1]^\mathbb{Z}$ (under the $\mathbb{Z}$-translation) is $1$, a \textit{finite} number. We shall have a more detailed explanation for mean dimension in Section 2, Section 3 and Section 4. This is an essential difference between $L(\mathbb{R})^\mathbb{N}$ and the Hilbert cube $([-1,1]^\mathbb{Z})^\mathbb{N}$. From this point of view, the ``size'' of the space $L(\mathbb{R})$ that we selected in the above solution is indeed too ``large''. Thus, we require a better candidate substantially. More precisely, we put a further problem:
\begin{itemize}\item
Is there an explicit \textit{compact} invariant subset $F\subset C(\mathbb{R})$ of \textit{finite} mean dimension such that $F^\mathbb{N}$ is universal?
\end{itemize}

\medskip

This problem is temporarily beyond the authors' reach. The aim of the present paper is to give a positive answer to a slightly weaker statement which strengthens Theorem \ref{gj} with a closer analogue of the Hilbert cube $([-1,1]^\mathbb{Z})^\mathbb{N}$ and a more direct construction of an embedding mapping, where we choose the spaces $\mathcal{B}([-n/3-1/2,-n/3]\cup[n/3,n/3+1/2])$ (for nonnegative integers $n$) because their mean dimension are equal to $1$, the same as the mean dimension of $[-1,1]^\mathbb{Z}$. We notice that the mean dimension of those spaces in Theorem \ref{gj} are not uniformly bounded by a finite number.
\begin{theorem}\label{submain}
Under the translation $\sigma$ the space $$\prod_{m=0}^{+\infty}\prod_{n=0}^m\mathcal{B}([-n/3-1/2,-n/3]\cup[n/3,n/3+1/2])$$ is universal.
\end{theorem}

\medskip

Furthermore, we have the following refinement which unifies the previously known results in this direction. Since for any compact $I\subset\mathbb{R}$ both $\mathcal{B}(I)$ and $L(\mathbb{R})$ are compact metric spaces with respect to the distance $D$, their intersection $\mathcal{B}(I)\cap L(\mathbb{R})$ is a compact metric space as well. Thus, $\mathcal{B}(I)\cap L(\mathbb{R})$ becomes also a compact invariant subset of $C(\mathbb{R})$ under the translation $\sigma$.

\begin{theorem}[Main theorem 1]\label{main}
For any real numbers $0<\alpha<\beta$ the space $$\prod_{m=0}^{+\infty}\prod_{n=0}^m\left(\mathcal{B}([-n\alpha-\beta,-n\alpha]\cup[n\alpha,n\alpha+\beta])\cap L(\mathbb{R})\right)$$ is universal under the translation $\sigma$.
\end{theorem}

\medskip

\begin{remark}
It is clear that Theorem \ref{submain} follows directly from Theorem \ref{main}. As a seemingly reachable question we may ask if it can be strengthened with a countable self-product of a member among them. Unfortunately, this is not correct. In fact, we shall clarify some wrong impression in Section 6:
\begin{itemize}
\item Under the translation $\sigma$ the space $\mathcal{B}([-1,1])^\mathbb{N}$ is \textit{not} universal.
\item Under the translation $\sigma$ the space $\prod_{n=0}^{+\infty}\mathcal{B}([-n/3-1/2,-n/3]\cup[n/3,n/3+1/2])$ is \textit{not} universal.
\end{itemize}
\end{remark}

We denote by $C^\infty(\mathbb{R})$ the set of smooth (i.e. infinitely differentiable) functions $f:\mathbb{R}\to[-1,1]$. Note that the space $C^\infty(\mathbb{R})\cap L(\mathbb{R})\subset C(\mathbb{R})$ is compact and invariant (under the translation $\sigma$).
\begin{corollary}
Under the translation $\sigma$ the space $\left(C^\infty(\mathbb{R})\cap L(\mathbb{R})\right)^\mathbb{N}$ is universal.
\end{corollary}

\medskip

As a quantitative result complementary to Theorem \ref{main}, we have the following classification of the real flows appearing in Theorem \ref{submain} under topological conjugacy and mean dimension. Note that for any positive real number $r$ the mean dimension of the translation on $\mathcal{B}([-r,r])\cap L(\mathbb{R})$ is finite. Therefore Theorem \ref{main} provides a more reasonable choice of a universal real flow.

\begin{theorem}[Main theorem 2]\label{maintheorembernstein}
Let $I,J\subset\mathbb{R}$ be compact intervals. Let $a\le b$ and $c\ge0$ be real numbers. The following assertions are true:
\begin{enumerate}
\item$(\mathcal{B}^\mathbb{C}(I),\mathbb{R},\sigma)$ is topologically conjugate to $(\mathcal{B}^\mathbb{C}(J),\mathbb{R},\sigma)$ if and only if $I=J$ or $I=-J$.
\item$(\mathcal{B}(I),\mathbb{R},\sigma)$ is topologically conjugate to $(\mathcal{B}(J),\mathbb{R},\sigma)$ if and only if $I\cap(-I)=J\cap(-J)$.
\item$\mdim(\mathcal{B}^\mathbb{C}([a,b]),\mathbb{R},\sigma)=2(b-a)$.
\item$\mdim(\mathcal{B}([-c,c]),\mathbb{R},\sigma)=2c$.
\end{enumerate}
\end{theorem}

\medskip

The definition of mean dimension is located in Section 2. We situate in Section 3 a short discussion about the universality of $L(\mathbb{R})$ (for $\mathbb{Z}$-actions), where we will indicate in particular that the mean dimension of $L(\mathbb{R})$ is $+\infty$. We will prove Theorem \ref{maintheorembernstein} in Section 4. In Section 5, we shall give a constructive proof of Theorem \ref{main}. A novelty of our method is that it overcomes a shortcoming arising from the Baire category approach. As presented in our proof, we are able to see an explicitly constructed embedding mapping of any real flow in the universal real flow that we suggested (which is also explicit) in the main theorem. Section 6 contains a collection of explanations in relation to our results (including the corollary), in which a final remark will end the body of this paper with the impossibility of improving our main result in a seemingly achievable direction. An appendix is logically independent of all the sections.

\bigskip

\noindent\textbf{Acknowledgements.}
The authors are grateful to the anonymous referee for his/her insightful comments and helpful suggestions which improve this paper greatly. L. Jin was supported by Basal Funding AFB 170001 and Fondecyt Grant No. 3190127, and was partially supported by NNSF of China No. 11971455. Y. Qiao was supported by NNSF of China No. 11901206. S. Tu was supported by NNSF of China No. 11801584 and 11871228.

\bigskip

\section{Preliminaries for mean dimension}
Mean dimension was introduced by Gromov \cite{G} in 1999, and was systematically studied by Lindenstrauss and Weiss \cite{LW} in 2000 as a topological invariant of dynamical systems. We collect in this section fundamental material (borrowed from \cite{LW,GJfm}) on mean dimension for $\mathbb{R}$-actions and $\mathbb{Z}$-actions.

\medskip

Let $(X,d)$ be a compact metric space and $\epsilon>0$. We denote by $\Widim_\epsilon(X,d)$ the minimum topological dimension $\dim(K)$ (i.e. the Lebesgue covering dimension) of a compact metrizable space $K$ which admits an $\epsilon$-embedding $f:X\to K$ with respect to the distance $d$, namely, a continuous mapping $f:X\to K$ satisfying that $f(x)=f(x^\prime)$ implies $d(x,x^\prime)<\epsilon$ for all $x,x^\prime\in X$. We may easily verify: $$\dim(X)=\lim_{\epsilon\to0}\Widim_\epsilon(X,d).$$

\medskip

Let $(X,\mathbb{R},T)$ be a real flow and $d$ a compatible metric on $X$. For each nonnegative real number $r$ we define a compatible metric $d_r^T$ on $X$ by $$d_r^T(x,x^\prime)=\max_{0\le t\le r}d(T_tx,T_tx^\prime),\quad\forall x,x^\prime\in X.$$ The \textbf{mean dimension} of $(X,\mathbb{R},T)$ is defined by $$\mdim(X,\mathbb{R},T)=\lim_{\epsilon\to0}\lim_{r\to+\infty}\frac{\Widim_\epsilon(X,d_r^T)}{r}.$$ The limits in the definition exist, and the value $\mdim(X,\mathbb{R},T)\in[0,+\infty]$ does not depend on the choice of a compatible metric $d$ on $X$. Clearly, if $(Y,\mathbb{R},S)$ can be embedded in $(X,\mathbb{R},T)$ then $$\mdim(Y,\mathbb{R},S)\le\mdim(X,\mathbb{R},T).$$

For every $\lambda\in\mathbb{R}$ we denote by $(X,\mathbb{R},T^\lambda)$ the real flow with the following action: $$T^\lambda:\mathbb{R}\times X\to X,\quad(t,x)\mapsto T_{\lambda t}x.$$
\begin{lemma}
For any real flow $(X,\mathbb{R},T)$ and $\lambda\in\mathbb{R}$ $$\mdim(X,\mathbb{R},T^\lambda)=|\lambda|\cdot\mdim(X,\mathbb{R},T).$$
\end{lemma}

\medskip

For a $\mathbb{Z}$-action $(X,\mathbb{Z},\phi)$, i.e. $(X,d)$ is a compact metric space and $\phi:X\to X$ is a homeomorphism, and a positive integer $n$ we define a compatible metric $d_n^\phi$ on $X$ by $$d_n^\phi(x,x^\prime)=\max_{i\in\mathbb{Z},\,0\le i\le n-1}d(\phi^ix,\phi^ix^\prime),\quad\forall x,x^\prime\in X.$$ We recall that the mean dimension of $(X,\mathbb{Z},\phi)$ is similarly defined by $$\mdim(X,\mathbb{Z},\phi)=\lim_{\epsilon\to0}\lim_{n\to+\infty}\frac{\Widim_\epsilon(X,d_n^\phi)}{n}.$$ The $\mathbb{Z}$-translation on $([0,1]^d)^\mathbb{Z}$ (where $d\in\mathbb{N}\cup\{+\infty\}$) is defined by $$\sigma:([0,1]^d)^\mathbb{Z}\to([0,1]^d)^\mathbb{Z},\quad(x_i)_{i\in\mathbb{Z}}\mapsto(x_{i+1})_{i\in\mathbb{Z}}.$$ We do not change the notation $\sigma$ here because there is no ambiguity (as the acting group is always indicated precisely in the midst of a triple).
\begin{lemma}
For every $d\in\mathbb{N}\cup\{+\infty\}$ $$\mdim(([0,1]^d)^\mathbb{Z},\mathbb{Z},\sigma)=d.$$
\end{lemma}

\medskip

We note that a real flow $(X,\mathbb{R},T)$ naturally induces a $\mathbb{Z}$-action $(X,\mathbb{Z},T_1)$.
\begin{lemma}
For any real flow $(X,\mathbb{R},T)$ $$\mdim(X,\mathbb{R},T)=\mdim(X,\mathbb{Z},T_1).$$
\end{lemma}

\medskip

We would like to remind the reader to keep in mind that all the statements in this section will be used \textit{implicitly} in this paper.

\medskip

\section{Universality of $L(\mathbb{R})$ for $\mathbb{Z}$-actions}
The purpose of this section is to show that under the translation $\sigma$ the space $L(\mathbb{R})$ has mean dimension $+\infty$. The main result of this section is Theorem \ref{thm:LR}.

\begin{theorem}\label{thm:LR}
All $\mathbb{Z}$-actions can be embedded in the $\mathbb{Z}$-translation on $L(\mathbb{R})$.
\end{theorem}
In other words, $(L(\mathbb{R}),\mathbb{Z},\sigma_1)$ is a universal $\mathbb{Z}$-action. It follows directly from Theorem \ref{thm:LR} that $$\mdim(L(\mathbb{R}),\mathbb{R},\sigma)=\mdim(L(\mathbb{R}),\mathbb{Z},\sigma_1)=+\infty.$$

\medskip

In order to prove Theorem \ref{thm:LR}, we employ a result for real flows, due to Gutman, Jin and Tsukamoto \cite[Theorem 1.3]{GJT}, which refines the classical Bebutov--Kakutani dynamical embedding theorem, as follows:
\begin{theorem}[{\cite{GJT}}]\label{thm:gjt}
A real flow $(X,\mathbb{R},T)$ can be embedded in $(L(\mathbb{R}),\mathbb{R},\sigma)$ if and only if the set of its fixed points $\{x\in X:T_tx=x,\,\forall t\in\mathbb{R}\}$ can be (topologically) embedded in $[0,1]$.
\end{theorem}

We remark here that it is also possible to give an elementary proof of Theorem \ref{thm:LR} (without applying Theorem \ref{thm:gjt}). In fact, it suffices to notice that $[0,1]^\mathbb{N}$ is (topologically) embedded in $$L_0(\mathbb{R}/\mathbb{Z})=\{f:\mathbb{R}/\mathbb{Z}\to[0,1]:\,f(0)=0,\,\;|f(x)-f(x^\prime)|\le|x-x^\prime|,\;\forall x,x^\prime\in\mathbb{R}/\mathbb{Z}\}.$$ However, we go through a more abstract approach as we would like to present a connection between real flows and $\mathbb{Z}$-actions. Now we prove Theorem \ref{thm:LR}.

\medskip

\begin{proof}
We take a $\mathbb{Z}$-action $(X,\mathbb{Z},\phi)$. Let $(S(X),\mathbb{R},T)$ be the suspension over $(X,\mathbb{Z},\phi)$ generated by the constant function $1$, namely, $$S(X)=(X\times[0,1])/\sim$$ where $\sim$ is the equivalence relation given by $(x,1)\sim(\phi(x),0)$, $$T:\mathbb{R}\times S(X)\to S(X),\quad(t,(x,s))\mapsto(\phi^n(x),s^\prime)$$ where $n\in\mathbb{N}$ and $s^\prime\in[0,1]$ satisfy $n+s^\prime=t+s$.

It is clear that the real flow $(S(X),\mathbb{R},T)$ has no fixed points. By Theorem \ref{thm:gjt} we know that $(S(X),\mathbb{R},T)$ can be embedded in $(L(\mathbb{R}),\mathbb{R},\sigma)$. Thus, $(X,\mathbb{Z},\phi)$ can be embedded in $(L(\mathbb{R}),\mathbb{Z},\sigma_1)$.
\end{proof}

\medskip

\begin{remark}
Theorems \ref{jt}, \ref{thm:LR} and \ref{thm:gjt} reveal a difference between $\mathbb{R}$-actions and $\mathbb{Z}$-actions. The space $L(\mathbb{R})$ (under the $\mathbb{Z}$-translation) encompasses all $\mathbb{Z}$-actions, while (under the $\mathbb{R}$-translation) it is far from universal for $\mathbb{R}$-actions. As we see in Theorem \ref{thm:gjt}, the fixed-point set of an $\mathbb{R}$-action is the only obstacle and decides its embeddability in $L(\mathbb{R})$. In particular, any $\mathbb{R}$-action containing no fixed points can be embedded in $L(\mathbb{R})$. Unfortunately, this could never happen between $\mathbb{Z}$-actions and the Hilbert cube $[0,1]^\mathbb{Z}$ (under the $\mathbb{Z}$-translation). Moreover it turns out that mean dimension \cite{G,LW} becomes crucial for the \textit{embedding problem} of $\mathbb{Z}$-actions. As we mentioned previously, a key difference is that $L(\mathbb{R})$ has \textit{infinite} mean dimension while the Hilbert cube $[0,1]^\mathbb{Z}$ has \textit{finite} mean dimension. We do not plan to involve more detailed explanation in this paper because it is not the main topic in relation to our purpose here. For the latest progress in this direction we refer to \cite{GT,GQT}. For $\mathbb{R}$-actions we can prove (with an argument essentially the same as in \cite{GJT,JT}) the following theorem:
\begin{itemize}\item
A real flow $(X,\mathbb{R},T)$ can be embedded in $(L(\mathbb{R})^N,\mathbb{R},\sigma)$, where $N\in\mathbb{N}\cup\{+\infty\}$, if and only if the set of its fixed points $\{x\in X:T_tx=x,\,\forall t\in\mathbb{R}\}$ can be (topologically) embedded in $[0,1]^N$.
\end{itemize}
\end{remark}

\medskip

\section{Mean dimension of Bernstein spaces}
The goal of this section is to show that under the translation $\sigma$ the space $\mathcal{B}([-r,r])$, and hence the space $\mathcal{B}([-r,r])\cap L(\mathbb{R})$, has finite mean dimension (where $r$ is a nonnegative real number). The main theorem of this section is Theorem \ref{bernsteins}.

\medskip

Let us start with necessary notions as a continuation of Section 1. For a tempered distribution $f$ its \textbf{Fourier transforms} $\mathcal{F}(f)$ and $\overline{\mathcal{F}}(f)$ are defined by $$\langle\mathcal{F}(f),g\rangle=\langle f,\overline{\mathcal{F}}(g)\rangle,\quad\langle\overline{\mathcal{F}}(f),g\rangle=\langle f,\mathcal{F}(g)\rangle,$$ where $g$ ranges over all rapidly decreasing functions. For example, we have $\mathcal{F}(0)=0$, $\mathcal{F}(e^{2\pi\sqrt{-1}\tau\cdot})=\delta_\tau$ (i.e. the delta probability measure at the point $\tau\in\mathbb{R}$), and $\overline{\mathcal{F}}(\mathcal{F}(f))=\mathcal{F}(\overline{\mathcal{F}}(f))=f$.

We would like to remark here that for any real-valued bounded continuous function $f$ on $\mathbb{R}$ we may verify that $\supp(\mathcal{F}(f))$ must be a \textit{symmetric} subset of $\mathbb{R}$. In fact, for a compact interval $I\subset\mathbb{R}$ and any $f\in\mathcal{B}(I)$ we have $$\supp(\mathcal{F}(f))=\supp(\mathcal{F}(\overline{f}))=-\supp(\mathcal{F}(f)).$$ Namely $\mathcal{B}(I)$ is equal to $\mathcal{B}(I\cap(-I))$. So it is better to fix the notation $\mathcal{B}([-c,c])$ (where $c\ge0$) for simplicity, and apparently, Theorem \ref{bernsteins} and Theorem \ref{maintheorembernstein} are equivalent.

\begin{remark}
For a compact interval $[a,b]\subset\mathbb{R}$ we usually call the space of (complex-valued) bounded continuous functions on $\mathbb{R}$ band-limited in $[a,b]$ a \textbf{Bernstein space}. This is a Banach space (with respect to the $L^\infty$-norm over $\mathbb{R}$). Strictly speaking, the title of this section is somewhat misleading, as $\mathcal{B}^\mathbb{C}([a,b])$ and $\mathcal{B}([-c,c])$ ($c\ge0$) themselves are not Bernstein spaces. But they are compact subsets of a Bernstein space. It is worth mentioning that we are interested in the compact metric spaces $\mathcal{B}^\mathbb{C}([a,b])$ and $\mathcal{B}([-c,c])$ because they have a background of deep applications to dynamical systems, and were heavily used in the \textit{embedding problem} of $\mathbb{Z}^k$-actions (for details see \cite{GT,GQT}). Giving a dynamical classification is valuable. The value of their mean dimension was first announced in \cite{GT} (without a proof).
\end{remark}

\medskip

We shall need brief preparations in front of the main result of this section. Our tools (borrowed from \cite{GT} with a slight modification in the statement) are sampling (Lemma \ref{samplingtheorem}) and interpolation (Lemma \ref{interpolatingtheorem}).
\begin{lemma}\label{samplingtheorem}
Suppose that two positive real numbers $a$ and $d$ satisfy $2ad<1$ and $f\in\mathcal{B}([-a,a])$. If $f(dn)=0$ for all $n\in\mathbb{Z}$, then $f\equiv0$.
\end{lemma}
\begin{lemma}\label{interpolatingtheorem}
For every $\epsilon>0$ there exists a rapidly decreasing function $f:\mathbb{R}\to\mathbb{R}$ band-limited in $[-(1+\epsilon)/2,(1+\epsilon)/2]$ such that $f(0)=1$ and $f(n)=0$ for all nonzero $n\in\mathbb{Z}$.
\end{lemma}

\medskip

\begin{theorem}\label{bernsteins}
Let $I,J,K,H\subset\mathbb{R}$ be compact intervals, where $K$ and $H$ are symmetric (i.e. $K=-K$ and $H=-H$). Let $a\le b$ and $c\ge0$ be real numbers. The following assertions are true:
\begin{enumerate}
\item$(\mathcal{B}^\mathbb{C}(I),\mathbb{R},\sigma)$ is topologically conjugate to $(\mathcal{B}^\mathbb{C}(J),\mathbb{R},\sigma)$ if and only if $I=J$ or $I=-J$.
\item$(\mathcal{B}(K),\mathbb{R},\sigma)$ is topologically conjugate to $(\mathcal{B}(H),\mathbb{R},\sigma)$ if and only if $K=H$.
\item$\mdim(\mathcal{B}^\mathbb{C}([a,b]),\mathbb{R},\sigma)=\mdim(\mathcal{B}^\mathbb{C}([a,b]),\mathbb{Z},\sigma_1)=2(b-a)$.
\item$\mdim(\mathcal{B}([-c,c]),\mathbb{R},\sigma)=\mdim(\mathcal{B}([-c,c]),\mathbb{Z},\sigma_1)=2c$.
\end{enumerate}
\end{theorem}
\begin{proof}
In the beginning of the proof we would like to remind the reader that the proofs of (4) and (3) are independent of (2) and (1), and we shall use (4) and (3) when proving (2) and (1).

\medskip

(2): The ``if'' part ``$\Longleftarrow$'' of the statement is obvious. To see the ``only if'' part ``$\Longrightarrow$'', we note that the topological conjugacy implies that the two real flows have the same mean dimension. By (4) and the fact that both $K$ and $H$ are symmetric compact intervals, we deduce $K=H$. The proof for (2) is simple (and moreover, it holds if and only if $|K|=|H|$ if and only if $\mdim(\mathcal{B}(K),\mathbb{R},\sigma)=\mdim(\mathcal{B}(H),\mathbb{R},\sigma)$), because a symmetric compact interval is fully decided by its length. However, this will be a problem for an arbitrary compact interval in the proof of (1).

\medskip

(1): We first note that the ``if'' part ``$\Longleftarrow$'' is easy, because if $I=J$ or $I=-J$ then $(\mathcal{B}^\mathbb{C}(I),\mathbb{R},\sigma)$ is topologically conjugate to $(\mathcal{B}^\mathbb{C}(J),\mathbb{R},\sigma)$ with the mapping: $f\mapsto f$ or $f\mapsto\overline{f}$. So we now prove the ``only if'' part ``$\Longrightarrow$''.

We suppose that $(\mathcal{B}^\mathbb{C}(I),\mathbb{R},\sigma)$ is topologically conjugate to $(\mathcal{B}^\mathbb{C}(J),\mathbb{R},\sigma)$. It follows that the two real flows must have the same value of mean dimension. According to (3) the compact intervals $I$ and $J$ must have the same length $|I|=|J|$. Let us assume $I\ne J$ and $I\ne-J$. This implies that there exist, without loss of generality, a positive real number $\tau\in I$ and a nonnegative real number $\gamma\in J$ satisfying that $\tau>\gamma$ and $J\subset[-\gamma,\gamma]$.

We consider the function $f(\cdot)=e^{2\pi\sqrt{-1}\tau\cdot}$ defined on $\mathbb{R}$. Note that $\mathcal{F}(f)=\delta_\tau$ which implies $f\in\mathcal{B}^\mathbb{C}(I)$. Clearly, $f$ is a periodic function, and its (fundamental) period is $T=1/\tau>0$. Thus, by equivariance, its image (under an embedding mapping) $g\in\mathcal{B}^\mathbb{C}(J)$ is also a periodic function such that $g(t+T)=g(t)$ for all $t\in\mathbb{R}$. We write the complex-valued function $g$ on $\mathbb{R}$ as $g=h+k\cdot\sqrt{-1}$, where $h$ and $k$ are real-valued functions on $\mathbb{R}$. Since $h=(g+\overline{g})/2$, we have $h\in\mathcal{B}(J\cup(-J))\subset\mathcal{B}([-\gamma,\gamma])$. The function $k=(g-\overline{g})/2\sqrt{-1}$ also belongs to $\mathcal{B}([-\gamma,\gamma])$. So the following argument, which deals with $h$, applies to $k$ as well.

Since $h\in\mathcal{B}([-\gamma,\gamma])$ and since $h(t+T)=h(t)$ for all $t\in\mathbb{R}$, the Fourier series representation of the periodic function $h$ on $\mathbb{R}$ (being in particular a restriction of a holomorphic function on $\mathbb{C}$) $$h(t)=\sum_{n=-\infty}^{+\infty}c_n\cdot e^{2\pi\sqrt{-1}nt/T},\quad\forall t\in\mathbb{R},$$ with $c_{-n}=\overline{c_n}$ for any $n\in\mathbb{Z}$, converges uniformly to $h$. In fact, we use here a generalization (for bounded continuous functions on the real line) of the classical Paley--Wiener theorem (we refer to Section 6 for its precise statement). It follows that $$\mathcal{F}(h)=c_0\mathcal{F}(1)+\sum_{n=1}^{+\infty}c_n\mathcal{F}(e^{2\pi\sqrt{-1}n\cdot/T})+\overline{c_n}\mathcal{F}(e^{-2\pi\sqrt{-1}n\cdot/T}).$$ Since $h$ is band-limited in $[-\gamma,\gamma]$, we have $c_n=0$ for any $n\in\mathbb{Z}$ with $|n|>\gamma T$. Since $0\le\gamma T=\gamma/\tau<1$, we finally deduce $c_n=0$ for all nonzero $n\in\mathbb{Z}$. This implies that $h$ is a constant function.

Similarly, so is $k$. Thus, we conclude that $g=h+k\cdot\sqrt{-1}$ (which is the image of $f$) must be a constant function. This, however, contradicts the injectivity of an embedding mapping.

Here we would like to remark shortly that by the sampling lemma we may show that the space $\mathcal{B}([-\gamma,\gamma])$ does not contain a function whose fundamental period is strictly less than $T/2$. But this is far from adequate for our argument (as the coefficient $1/2$ with $T$ is unpleasant). So we have to go through Fourier series representations of those functions in $\mathcal{B}([-\gamma,\gamma])$ in the proof.

\bigskip

(4) \& (3): First of all, we should note that both $\mathcal{B}^\mathbb{C}(\{0\})$ and $\mathcal{B}(\{0\})$ contain only constant functions, and thus, their mean dimension (under the translation $\sigma$) are equal to zero. As follows we build three lemmas which will reduce the statement of (4) and (3) to a standard case.
\begin{lemma}
For any $c>0$ and any $\lambda>0$ $$\mdim(\mathcal{B}([-\lambda c,\lambda c]),\mathbb{R},\sigma)=\lambda\cdot\mdim(\mathcal{B}([-c,c]),\mathbb{R},\sigma).$$
\end{lemma}
\begin{proof}
We omit the proof of this lemma because it is absolutely the same as the lemma below. A bridge is $\mdim(\mathcal{B}([-c,c]),\mathbb{R},\sigma^\lambda)$.
\end{proof}
\begin{lemma}
For any $a<b$ and any $\lambda>0$ $$\mdim(\mathcal{B}^\mathbb{C}([\lambda a,\lambda b]),\mathbb{R},\sigma)=\lambda\cdot\mdim(\mathcal{B}^\mathbb{C}([a,b]),\mathbb{R},\sigma).$$
\end{lemma}
\begin{proof}
For each $a<b$ and each $\lambda>0$ we consider the mapping: $$\rho:\mathcal{B}^\mathbb{C}([a,b])\to\mathcal{B}^\mathbb{C}([\lambda a,\lambda b]),\quad f(t)\mapsto f(\lambda t).$$ Note here that for any bounded continuous function $f:\mathbb{R}\to\mathbb{C}$ band-limited in $[a,b]$ we have $$\supp(\mathcal{F}(f(\lambda\cdot)))\subset\lambda\supp(\mathcal{F}(f))\subset[\lambda a,\lambda b].$$ Clearly, the mapping $\rho$ is continuous and injective, and hence (by compactness) is a topological embedding mapping. Moreover, $\rho$ is equivariant because we have for any $f\in\mathcal{B}^\mathbb{C}([a,b])$ and any $r\in\mathbb{R}$ $$\rho\circ(\sigma^\lambda)_rf(t)=\rho\circ\sigma_{\lambda r}f(t)=f(\lambda t+\lambda r)=f(\lambda(t+r))=\sigma_r\circ\rho f(t),\quad\forall t\in\mathbb{R}.$$ It follows that $(\mathcal{B}^\mathbb{C}([a,b]),\mathbb{R},\sigma^\lambda)$ can be embedded in $(\mathcal{B}^\mathbb{C}([\lambda a,\lambda b]),\mathbb{R},\sigma)$. This implies that $$\lambda\cdot\mdim(\mathcal{B}^\mathbb{C}([a,b]),\mathbb{R},\sigma)=\mdim(\mathcal{B}^\mathbb{C}([a,b]),\mathbb{R},\sigma^\lambda)\le\mdim(\mathcal{B}^\mathbb{C}([\lambda a,\lambda b]),\mathbb{R},\sigma).$$ This inequality also applies to $\lambda a<\lambda b$ and $1/\lambda>0$ (instead of $a<b$ and $\lambda>0$, respectively). Thus we conclude $$\mdim(\mathcal{B}^\mathbb{C}([\lambda a,\lambda b]),\mathbb{R},\sigma)=\mdim(\mathcal{B}^\mathbb{C}([a,b]),\mathbb{R},\sigma^\lambda)=\lambda\cdot\mdim(\mathcal{B}^\mathbb{C}([a,b]),\mathbb{R},\sigma).$$
\end{proof}
\begin{lemma}
For every $a<b$ and every $\tau\in\mathbb{R}$ $$\mdim(\mathcal{B}^\mathbb{C}([a+\tau,b+\tau]),\mathbb{R},\sigma)=\mdim(\mathcal{B}^\mathbb{C}([a,b]),\mathbb{R},\sigma).$$
\end{lemma}
\begin{proof}
If $\tau=0$ then the statement is trivial. So we fix $a<b$ and a nonzero real number $\tau$ arbitrarily. We first notice that for any $f\in\mathcal{B}^\mathbb{C}([a,b])$ the function $e^{2\pi\sqrt{-1}\cdot}f$ is band-limited in $$\supp(\mathcal{F}(e^{2\pi\sqrt{-1}\cdot})\ast\mathcal{F}(f))\subset\supp(\delta_1)+\supp(\mathcal{F}(f))\subset[1+a,1+b].$$ Since $(\mathcal{B}^\mathbb{C}([a,b]),\mathbb{Z},\sigma_1)$ can be embedded in $(\mathcal{B}^\mathbb{C}([a+1,b+1]),\mathbb{Z},\sigma_1)$ (which are $\mathbb{Z}$-actions) with the mapping $$\mathcal{B}^\mathbb{C}([a,b])\to\mathcal{B}^\mathbb{C}([a+1,b+1]),\quad f(t)\mapsto e^{2\pi\sqrt{-1}t}f(t),$$ which also applies to $a+1<b+1$ and $-1$ (instead of $a<b$ and $1$, respectively), we deduce $$\mdim(\mathcal{B}^\mathbb{C}([a,b]),\mathbb{Z},\sigma_1)=\mdim(\mathcal{B}^\mathbb{C}([a+1,b+1]),\mathbb{Z},\sigma_1)$$ which implies $$\mdim(\mathcal{B}^\mathbb{C}([a,b]),\mathbb{R},\sigma)=\mdim(\mathcal{B}^\mathbb{C}([a+1,b+1]),\mathbb{R},\sigma).$$ In contrast to the above two lemmas (where the proof showed an $\mathbb{R}$-equivariant embedding mapping), here the point is that $e^{2\pi\sqrt{-1}\cdot}$ is a periodic function and its fundamental period is $1$, which allows us to obtain an embedding mapping which is $\mathbb{Z}$-equivariant rather than $\mathbb{R}$-equivariant. Without loss of generality we now assume $\tau>0$. Applying the above equality together with the previous lemma we conclude
\begin{align*}
&\mdim(\mathcal{B}^\mathbb{C}([a+\tau,b+\tau]),\mathbb{R},\sigma)\\
=&\tau\cdot\mdim(\mathcal{B}^\mathbb{C}([1+a/\tau,1+b/\tau]),\mathbb{R},\sigma)\\
=&\tau\cdot\mdim(\mathcal{B}^\mathbb{C}([a/\tau,b/\tau]),\mathbb{R},\sigma)\\
=&\mdim(\mathcal{B}^\mathbb{C}([a,b]),\mathbb{R},\sigma).
\end{align*}
\end{proof}

\medskip

Thus, it suffices to prove the following assertion: $$\mdim(\mathcal{B}^\mathbb{C}([0,1/2]),\mathbb{Z},\sigma_1)=\mdim(\mathcal{B}([-1/2,1/2]),\mathbb{Z},\sigma_1)=1.$$
So in the remaining part of the proof we are going to show the following three inequalities, respectively, which will end our argument:
$$\mdim(\mathcal{B}^\mathbb{C}([0,1/2]),\mathbb{Z},\sigma_1)\le\mdim(\mathcal{B}([-1/2,1/2]),\mathbb{Z},\sigma_1),$$
$$\mdim(\mathcal{B}([-1/2,1/2]),\mathbb{Z},\sigma_1)\le1,$$
$$\mdim(\mathcal{B}^\mathbb{C}([0,1/2]),\mathbb{Z},\sigma_1)\ge1.$$

\medskip

To show the first inequality, we take $0<\epsilon<1/2$ arbitrarily. We consider a continuous mapping: $$H_\epsilon:\mathcal{B}^\mathbb{C}([\epsilon,1/2])\to\mathcal{B}([-1/2,1/2]),\quad f\mapsto\frac12(f+\overline{f}).$$ Note that for any $f\in\mathcal{B}^\mathbb{C}([\epsilon,1/2])$ the function $(f+\overline{f})/2$ is real-valued and is band-limited in $$\supp(\mathcal{F}(f)+\mathcal{F}(\overline{f}))\subset\supp(\mathcal{F}(f))\cup(-\supp(\mathcal{F}(f)))\subset[-1/2,-\epsilon]\cup[\epsilon,1/2].$$ Clearly, $H_\epsilon$ is equivariant, i.e. it satisfies $H_\epsilon\circ\sigma_1=\sigma_1\circ H_\epsilon$. For every $f,g\in\mathcal{B}^\mathbb{C}([\epsilon,1/2])$ with $H_\epsilon(f)=H_\epsilon(g)$ we have $f-g=\overline{g-f}$. Since $\supp(\mathcal{F}(f-g))\subset[\epsilon,1/2]$ and $\supp(\mathcal{F}(\overline{g-f}))\subset[-1/2,-\epsilon]$, we get $\mathcal{F}(f-g)=0$ and hence $f=g$. Therefore $H_\epsilon$ is injective. This shows that $(\mathcal{B}^\mathbb{C}([\epsilon,1/2]),\mathbb{Z},\sigma_1)$ can be embedded in $(\mathcal{B}([-1/2,1/2]),\mathbb{Z},\sigma_1)$ (by the mapping $H_\epsilon$). It follows that $$\mdim(\mathcal{B}^\mathbb{C}([\epsilon,1/2]),\mathbb{Z},\sigma_1)\le\mdim(\mathcal{B}([-1/2,1/2]),\mathbb{Z},\sigma_1)$$ which implies $$\mdim(\mathcal{B}^\mathbb{C}([0,(1-2\epsilon)/2]),\mathbb{Z},\sigma_1)\le\mdim(\mathcal{B}([-1/2,1/2]),\mathbb{Z},\sigma_1).$$ Thus, $$(1-2\epsilon)\cdot\mdim(\mathcal{B}^\mathbb{C}([0,1/2]),\mathbb{Z},\sigma_1)\le\mdim(\mathcal{B}([-1/2,1/2]),\mathbb{Z},\sigma_1).$$ Since $0<\epsilon<1/2$ is arbitrary, we get the first inequality.

\medskip

In order to show the second inequality, we need employ the sampling lemma. We fix an arbitrary $0<b<1/2$ and consider a continuous mapping: $$\Lambda_b:\mathcal{B}([-b,b])\to[-1,1]^\mathbb{Z},\quad f\mapsto f|_\mathbb{Z}.$$ Obviously, $\Lambda_b$ is $\mathbb{Z}$-equivariant, i.e. satisfying $\Lambda_b\circ\sigma_1=\sigma\circ\Lambda_b$. By Lemma \ref{samplingtheorem}, the mapping $\Lambda_b$ is injective. Thus, $(\mathcal{B}([-b,b]),\mathbb{Z},\sigma_1)$ can be embedded in $([-1,1]^\mathbb{Z},\mathbb{Z},\sigma)$ (with the mapping $\Lambda_b$). It follows that $$\mdim(\mathcal{B}([-b,b]),\mathbb{Z},\sigma_1)\le\mdim([-1,1]^\mathbb{Z},\mathbb{Z},\sigma)$$ which implies $$2b\cdot\mdim(\mathcal{B}([-1/2,1/2]),\mathbb{Z},\sigma_1)\le\mdim([-1,1]^\mathbb{Z},\mathbb{Z},\sigma)=1.$$ Since $0<b<1/2$ is arbitrary, we obtain the second inequality.

\bigskip

Finally, to prove the third inequality we shall need the interpolation lemma. We fix an arbitrary $\epsilon>0$. By Lemma \ref{interpolatingtheorem}, there is a rapidly decreasing function $f:\mathbb{R}\to\mathbb{R}$ band-limited in $[-(1+\epsilon)/2,(1+\epsilon)/2]$ satisfying that $f(0)=1$ and $f(n)=0$ for all $n\in\mathbb{Z}\setminus\{0\}$. In particular, $$|f(t)|\le\frac{C}{1+t^2},\quad\forall t\in\mathbb{R},$$ for some constant $C>0$. Set $$K=\max_{t\in\mathbb{R}}\sum_{n\in\mathbb{Z}}\frac{C}{1+(t-n)^2}<+\infty.$$ We define a mapping as follows: $$G_\epsilon:([0,1]^2)^\mathbb{Z}\to\mathcal{B}^\mathbb{C}([-(1+\epsilon)/2,(1+\epsilon)/2]),$$$$a=(a_{1,n},a_{2,n})_{n\in\mathbb{Z}}\mapsto G_\epsilon(a),$$$$G_\epsilon(a)(t)=\frac{1}{2K}\sum_{n\in\mathbb{Z}}(a_{1,n}+a_{2,n}\sqrt{-1})f(t-n),\quad\forall t\in\mathbb{R}.$$ It is clear that $G_\epsilon$ is $\mathbb{Z}$-equivariant, namely $G_\epsilon\circ\sigma=\sigma_1\circ G_\epsilon$.

To prove the injectivity of $G_\epsilon$, we take $a=(a_{1,n},a_{2,n})_{n\in\mathbb{Z}},b=(b_{1,n},b_{2,n})_{n\in\mathbb{Z}}\in([0,1]^2)^\mathbb{Z}$ and assume $G_\epsilon(a)=G_\epsilon(b)$ which means $$\sum_{n\in\mathbb{Z}}\left((a_{1,n}-b_{1,n})+(a_{2,n}-b_{2,n})\sqrt{-1}\right)\cdot f(t-n)=0,\quad\forall t\in\mathbb{R}.$$ For every $m\in\mathbb{Z}$ by letting $t=m$ in the above equality we have $a_{1,m}=b_{1,m}$ and $a_{2,m}=b_{2,m}$. Thus we conclude $a=b$, which shows that $G_\epsilon$ is injective.

To see the continuity of $G_\epsilon$, we fix $a=(a_{1,n},a_{2,n})_{n\in\mathbb{Z}}\in([0,1]^2)^\mathbb{Z}$. We take a sequence $\{a^{(k)}=(a^{(k)}_{1,n},a^{(k)}_{2,n})_{n\in\mathbb{Z}}\}_{k\in\mathbb{N}}$ in $([0,1]^2)^\mathbb{Z}$ and suppose $a^{(k)}\to a$ as $k\to\infty$. We fix $N\in\mathbb{N}$ arbitrarily. For any $\delta>0$ there exists $m\in\mathbb{N}$ sufficiently large such that $$\frac{C}{K}\cdot\sum_{|n|>m}\frac{1}{1+(t-n)^2}<\frac{\delta}{2},\quad\forall t\in[-N,N].$$ For such an $m\in\mathbb{N}$ there is $l\in\mathbb{N}$ sufficiently large satisfying $$|\!|a^{(k)}-a|\!|_{l^\infty([-m,m])}<\delta$$ for all $k\ge l$. Therefore we have for any $k\ge l$ and any $t\in[-N,N]$
\begin{align*}
&\left|G_\epsilon(a^{(k)})(t)-G_\epsilon(a)(t)\right|\\
\le&\frac{1}{2K}\sum_{n\in\mathbb{Z}}|(a^{(k)}_{1,n}-a_{1,n})+(a^{(k)}_{2,n}-a_{2,n})\sqrt{-1}|\cdot|f(t-n)|\\
\le&\frac{1}{2K}\cdot|\!|a^{(k)}-a|\!|_{l^\infty([-m,m])}\cdot\sum_{|n|\le m}|f(t-n)|+\frac{1}{K}\cdot\sum_{|n|>m}|f(t-n)|\\
\le&\frac{1}{2K}\cdot|\!|a^{(k)}-a|\!|_{l^\infty([-m,m])}\cdot\sum_{n\in\mathbb{Z}}\frac{C}{1+(t-n)^2}+\frac{C}{K}\cdot\sum_{|n|>m}\frac{1}{1+(t-n)^2}\\
\le&\frac{1}{2}\cdot|\!|a^{(k)}-a|\!|_{l^\infty([-m,m])}+\frac{\delta}{2}<\delta.
\end{align*}
This implies $$\lim_{k\to\infty}|\!|G_\epsilon(a^{(k)})-G_\epsilon(a)|\!|_{L^\infty([-N,N])}=0.$$ Since $N\in\mathbb{N}$ is arbitrary, we deduce that $G_\epsilon$ is continuous.

Thus, $(([0,1]^2)^\mathbb{Z},\mathbb{Z},\sigma)$ can be embedded in $(\mathcal{B}^\mathbb{C}([-(1+\epsilon)/2,(1+\epsilon)/2]),\mathbb{Z},\sigma_1)$ (with the mapping $G_\epsilon$). It follows that $$\mdim(([0,1]^2)^\mathbb{Z},\mathbb{Z},\sigma)\le\mdim(\mathcal{B}^\mathbb{C}([-(1+\epsilon)/2,(1+\epsilon)/2]),\mathbb{Z},\sigma_1).$$ As an immediate consequence
\begin{align*}
&2(1+\epsilon)\cdot\mdim(\mathcal{B}^\mathbb{C}([0,1/2]),\mathbb{Z},\sigma_1)\\
=&2(1+\epsilon)\cdot\mdim(\mathcal{B}^\mathbb{C}([-1/4,1/4]),\mathbb{Z},\sigma_1)\\
=&\mdim(\mathcal{B}^\mathbb{C}([-(1+\epsilon)/2,(1+\epsilon)/2]),\mathbb{Z},\sigma_1)\\
\ge&\mdim(([0,1]^2)^\mathbb{Z},\mathbb{Z},\sigma)=2.
\end{align*}
Since $\epsilon>0$ is arbitrary, we obtain $$\mdim(\mathcal{B}^\mathbb{C}([0,1/2]),\mathbb{Z},\sigma_1)\ge1.$$ This completes the proof.
\end{proof}

\medskip

\section{Construction of an explicit embedding mapping}
\subsection{Overview}
The aim of this section is to prove Theorem \ref{main}. The proof goes through an approach of harmonic analysis. Our strategy demonstrates an application of classical analysis to topological dynamical systems. It would be worth mentioning that \textit{convolution} will be used frequently in our method, which brings about a \textit{dynamical} interaction, with \textit{visible embedding mappings} to the reader, between \textit{abstract} topological objects and \textit{concrete} function spaces possessed of some good nature (i.e. expressing all the abstract flows as a family of uniformly bounded analytic functions having a uniformly band-limited property and sharing a uniform Lipschitz constant). Our technique has some novelty substantially different from the Baire category framework (which, in particular, was extensively applied when embedding a class of actions in a \textit{finite} mean dimensional space), and applies to an example of a countable (infinite) product of finite mean dimensional spaces (which, however, seems to be hopeless at dealing with a finite product of finite mean dimensional spaces).

To begin with, let us fix a real flow $(X,\mathbb{R},T)$. We shall embed $(X,\mathbb{R},T)$ in the universal real flow $\left(\prod_{m=0}^{+\infty}\prod_{n=0}^m\left(\mathcal{B}([-n/2-1,-n/2]\cup[n/2,n/2+1])\cap L(\mathbb{R})\right),\mathbb{R},\sigma\right)$ with an explicitly built embedding mapping. We assume here $\alpha=1/2$ and $\beta=1$ only for the sake of convenience, as to which the proof for any positive real numbers $\alpha<\beta$ is being similar.

The procedure is going to be fulfilled within three steps. We shrink in each step the universal space, while preserving all the required properties, established in the previous step. The task of each step is indicated precisely in its title. The following diagram is a sketch of the route:

$$X\longrightarrow\textrm{ (a compact invariant subset of) }C(\mathbb{R})^\mathbb{N}$$$$\longrightarrow\prod_{m=0}^{+\infty}\prod_{n=0}^m\mathcal{B}([-n/2-1,-n/2]\cup[n/2,n/2+1])$$$$\longrightarrow\prod_{m=0}^{+\infty}\prod_{n=0}^m\left(\mathcal{B}([-n/2-1,-n/2]\cup[n/2,n/2+1])\cap L(\mathbb{R})\right)$$$$$$$$\textrm{Figure 1: Outline.}$$

\bigskip

\noindent\textbf{Setting.}
We denote by $\textbf{D}$ a compatible metric on the product space $C(\mathbb{R})^\mathbb{N}$: $$\textbf{D}(f,g)=\sum_{n\in\mathbb{N}}\sum_{N\in\mathbb{N}}\frac{|\!|f_n-g_n|\!|_{L^\infty([-N,N])}}{2^{n+N}},$$ where $f=(f_n)_{n\in\mathbb{N}},g=(g_n)_{n\in\mathbb{N}}\in C(\mathbb{R})^\mathbb{N}$. We note that each element in $C(\mathbb{R})^\mathbb{N}$ is identified with a continuous function $f:\mathbb{R}\to[-1,1]^\mathbb{N}$. Moreover, when dealing with complex-valued continuous functions on $\mathbb{R}$ we shall automatically adapt the distance $\textbf{D}$ to the complex context.

\medskip

\subsection{Embedding $(X,\mathbb{R},T)$ in $(C(\mathbb{R})^\mathbb{N},\mathbb{R},\sigma)$}
The way to this target is standard. We reproduce it in this step for completeness.

Let $\phi:X\to\phi(X)\subset[0,1]^\mathbb{N}$ be a homeomorphism (i.e. a topological embedding) of the compact metric space $X$ into $[0,1]^\mathbb{N}$. We define a mapping $$\Phi_1:X\to C(\mathbb{R})^\mathbb{N},\quad x\mapsto\Phi_1(x);\quad\quad\Phi_1(x)(t)=\phi(T_tx),\quad\forall t\in\mathbb{R}.$$ Note here that for every $x\in X$ its image $\Phi_1(x):\mathbb{R}\to[0,1]^\mathbb{N}$ is indeed a continuous function.

Obviously, the mapping $\Phi_1:X\to C(\mathbb{R})^\mathbb{N}$ is injective. To see that $\Phi_1$ is continuous, we take a sequence of points $x_n$ ($n\in\mathbb{N}$) in $X$, tending to some $x\in X$ as $n\to+\infty$, and a compact subset $A$ of $\mathbb{R}$. When $n\in\mathbb{N}$ is large enough, the distance between any two points $(t,x_n)$ and $(t,x)$ in $A\times X$ is sufficiently close to zero. Since $A\times X$ and $X$ are compact, $T$ and $\phi$ are uniformly continuous on $A\times X$ and $X$, respectively. It follows that the distance between $T_tx_n$ and $T_tx$ in $X$ is sufficiently close to zero as well, for all $t\in A$. This implies that $\Phi_1(x_n)(t)=\phi(T_tx_n)$ is sufficiently close to $\Phi_1(x)(t)=\phi(T_tx)$ for all $t\in A$. Thus, the sequence of continuous functions $\Phi_1(x_n)$ converges uniformly to $\Phi_1(x)$ on $A\subset\mathbb{R}$ as $n\to+\infty$, which shows that the mapping $\Phi_1$ is continuous. Since $X$ is compact and since $\Phi_1$ is continuous and one-to-one, $\Phi_1:X\to\Phi_1(X)$ is a homeomorphism.

For every $r\in\mathbb{R}$ and every $x\in X$ we have $$\Phi_1(T_rx)(t)=\phi(T_t(T_rx))=\phi(T_{t+r}x)=\Phi_1(x)(t+r)=\sigma_r(\Phi_1(x))(t)$$ for all $t\in\mathbb{R}$, which means $\Phi_1\circ T_r=\sigma_r\circ\Phi_1$ for any $r\in\mathbb{R}$, i.e. $\Phi_1$ is equivariant. Therefore $(X,\mathbb{R},T)$ can be embedded in $(C(\mathbb{R})^\mathbb{N},\mathbb{R},\sigma)$ with the mapping $\Phi_1$.

\medskip

\subsection{Embedding the translation on any compact invariant subset of $C(\mathbb{R})^\mathbb{N}$ in $\left(\prod_{m=0}^{+\infty}\prod_{n=0}^m\mathcal{B}([-n/2-1,-n/2]\cup[n/2,n/2+1]),\mathbb{R},\sigma\right)$}
This step will be accomplished with the following construction and lemmas. As the space $C(\mathbb{R})^\mathbb{N}$ is not compact, we shall embed any of its compact invariant subsets (namely $\Phi_1(X)$) rather than itself. In fact, embedding $C(\mathbb{R})^\mathbb{N}$ itself is also achievable with a little bit more effort. Nevertheless, we need not deal with it. Now let us fix a compact invariant subset of $C(\mathbb{R})^\mathbb{N}$, which we denote still by $X$ (instead of $\Phi_1(X)$).

\medskip

For every $n\in\mathbb{Z}$ we take a continuous function $\xi_n:\mathbb{R}\to[0,1]$ defined by $$\xi_n(t+n/2)=\begin{cases}0,\quad&t\in(-\infty,-1/2)\cup(1/2,+\infty)\\1+2t,\quad&t\in[-1/2,0]\\1-2t,\quad&t\in[0,1/2]\end{cases}$$ and let $\varphi_n=\overline{\mathcal{F}}(\xi_n)$. For any integer $n$ it is clear that $\varphi_n:\mathbb{R}\to\mathbb{C}$ is a bounded continuous function and satisfies $$\supp(\mathcal{F}(\varphi_n))=\supp(\mathcal{F}(\overline{\mathcal{F}}(\xi_n)))=\supp(\xi_n)=[(n-1)/2,(n+1)/2].$$ Moreover, we may verify $$\sum_{n=-\infty}^{+\infty}\xi_n(t)=1,\quad\forall t\in\mathbb{R}.$$ We set for each $n\in\mathbb{Z}$ $$k_n=\int_{-\infty}^{+\infty}\left|\varphi_n(s)\right|ds=\int_{-\infty}^{+\infty}\left|\overline{\mathcal{F}}(\xi_n)(s)\right|ds.$$ Notice that $$0<k_n<+\infty,\quad\forall n\in\mathbb{Z}.$$ We guarantee that each $k_n$ is positive because $\xi_n\in C(\mathbb{R})$ is not the constant function zero, while each $k_n$ is a finite number because the continuous function $\xi_n:\mathbb{R}\to[0,1]$ is supported in a compact interval.

\medskip

Note that here we stick to the case $\beta=2\alpha>0$. In general, for $0<\alpha<\beta$ the intervals $(n\alpha,n\alpha+\beta)$, where $n$ ranges over $\mathbb{Z}$, still form a countable open cover of the real line $\mathbb{R}$. If $\beta<2\alpha$ then the roof of any tent function $\xi_n$ that we chose should be replaced by a segment (instead of a point) at the level $1$. If $\beta>2\alpha$, then there are \textit{removable} elements in the countable open cover of $\mathbb{R}$, and we thus select a subcover of $\mathbb{R}$, from which, we need ensure that no members can be removed.

As we mentioned previously, convolution will be used frequently in our method. Such an operation is (quite) mild and is (in some sense) easy to control. In particular, it has some ``low-pass filter'' nature. A simple but important fact is that it is able to dominate the support of the Fourier transform (under this operation) with that of the participants. Besides, we recall that for any bounded continuous function $f:\mathbb{R}\to\mathbb{R}$ the set $\supp(\mathcal{F}(f))$ must be symmetric in $\mathbb{R}$. So we adopt a process with the help of $\mathcal{B}^\mathbb{C}(\cdot)$ which is more flexible than $\mathcal{B}(\cdot)$.

\medskip

\begin{lemma}
For any $n\in\mathbb{Z}$ and any $h\in C(\mathbb{R})$ we have $$\frac{h\ast\varphi_n}{k_n}\in\mathcal{B}^\mathbb{C}([(n-1)/2,(n+1)/2]),$$ where $h\ast\varphi_n$ denotes the convolution: $$h\ast\varphi_n(t)=\int_{-\infty}^{+\infty}h(t-s)\varphi_n(s)ds,\quad\forall t\in\mathbb{R}.$$
\end{lemma}
\begin{proof}
We fix an integer $n$ and a continuous function $h:\mathbb{R}\to[-1,1]$. Clearly, the function $h\ast\varphi_n:\mathbb{R}\to\mathbb{C}$ is continuous. Since $|\!|h|\!|_{L^\infty(\mathbb{R})}\le1$, we have for all $t\in\mathbb{R}$ $$\left|h\ast\varphi_n(t)\right|=\left|\int_{-\infty}^{+\infty}h(t-s)\varphi_n(s)ds\right|\le\int_{-\infty}^{+\infty}|\varphi_n(s)|ds=k_n$$ which implies $$\left|\!\left|\frac{h\ast\varphi_n}{k_n}\right|\!\right|_{L^\infty(\mathbb{R})}\le1.$$ Since $\supp(\mathcal{F}(\varphi_n))=[(n-1)/2,(n+1)/2]$, we deduce $$\supp(\mathcal{F}(h\ast\varphi_n))\subset\supp(\mathcal{F}(h))\cap\supp(\mathcal{F}(\varphi_n))\subset[(n-1)/2,(n+1)/2],$$ as required.
\end{proof}

\medskip

We now define a mapping $$\Phi_2:C(\mathbb{R})^\mathbb{N}\to\prod_{(l,n)\in\mathbb{N}\times\mathbb{Z}}\mathcal{B}^\mathbb{C}([(n-1)/2,(n+1)/2]),$$$$(f_l)_{l\in\mathbb{N}}\mapsto\left(\frac{f_l\ast\varphi_n}{k_n}\right)_{(l,n)\in\mathbb{N}\times\mathbb{Z}}.$$ We shall prove that $\Phi_2:X\to\Phi_2(X)$ is an equivariant homeomorphism, which will imply that $(X,\mathbb{R},\sigma)$ can be embedded in $\left(\prod_{(l,n)\in\mathbb{N}\times\mathbb{Z}}\mathcal{B}^\mathbb{C}([(n-1)/2,(n+1)/2]),\mathbb{R},\sigma\right)$ with the mapping $\Phi_2$.

\medskip

\begin{lemma}
The mapping $\Phi_2:X\to\Phi_2(X)$ is an equivariant homeomorphism.
\end{lemma}
\begin{proof}
For any $n\in\mathbb{Z}$, any $h\in C(\mathbb{R})$ and any $r\in\mathbb{R}$
\begin{align*}
\sigma_r(h\ast\varphi_n)(t)&=\int_{-\infty}^{+\infty}h(t+r-s)\varphi_n(s)ds\\&=\int_{-\infty}^{+\infty}\sigma_rh(t-s)\varphi_n(s)ds\\&=(\sigma_rh)\ast\varphi_n(t)
\end{align*}
for all $t\in\mathbb{R}$. This shows that $\sigma_r\circ\Phi_2=\Phi_2\circ\sigma_r$ for every $r\in\mathbb{R}$.

We note that to show that $\Phi_2:X\to\Phi_2(X)$ is a homeomorphism, it suffices to prove that $\Phi_2$ is continuous and injective (because $X$ is a compact metric space).

To see that $\Phi_2$ is continuous, we take $f=(f_j)_{j\in\mathbb{N}}\in C(\mathbb{R})^\mathbb{N}$. We fix $(i,n)\in\mathbb{N}\times\mathbb{Z}$ arbitrarily. For any $\epsilon>0$ we choose:
\begin{itemize}
\item $A>0$ sufficiently large such that $$\int_{\mathbb{R}\setminus[-A,A]}\left|\varphi_n(t)\right|dt<\epsilon/8;$$
\item $N\in\mathbb{N}$ sufficiently large such that $$|\!|u-v|\!|_{L^\infty([-N,N])}<\epsilon/2\Longrightarrow D(u,v)<\epsilon,\quad\forall u,v\in C(\mathbb{R});$$
\item $\delta>0$ sufficiently small such that $$\textbf{D}(f,g)<\delta\Longrightarrow|\!|f_i-g_i|\!|_{L^\infty([-A-N,A+N])}<\epsilon/4k_n,\quad\forall g=(g_j)_{j\in\mathbb{N}}\in C(\mathbb{R})^\mathbb{N}.$$
\end{itemize}
It follows that if $g=(g_j)_{j\in\mathbb{N}}\in C(\mathbb{R})^\mathbb{N}$ satisfies $\textbf{D}(f,g)<\delta$ then for any $t\in[-N,N]$ we have
\begin{align*}
\left|f_i\ast\varphi_n(t)-g_i\ast\varphi_n(t)\right|&\le\int_{-\infty}^{+\infty}\left|f_i(t-s)-g_i(t-s)\right|\cdot\left|\varphi_n(s)\right|ds\\
&\le2\int_{\mathbb{R}\setminus[-A,A]}\left|\varphi_n(s)\right|ds+\int_{-A}^A\left|f_i(t-s)-g_i(t-s)\right|\cdot\left|\varphi_n(s)\right|ds\\
&\le2\cdot(\epsilon/8)+(\epsilon/4)\cdot\frac{1}{k_n}\int_{-A}^A\left|\varphi_n(s)\right|ds<\epsilon/2
\end{align*}
which implies $$|\!|f_i\ast\varphi_n-g_i\ast\varphi_n|\!|_{L^\infty([-N,N])}<\epsilon/2.$$ Thus, $$D(f_i\ast\varphi_n,g_i\ast\varphi_n)<\epsilon.$$ Since $(i,n)\in\mathbb{N}\times\mathbb{Z}$ and the function $f\in C(\mathbb{R})^\mathbb{N}$ are arbitrary, we deduce that $\Phi_2$ is continuous.

To verify that $\Phi_2$ is injective, we take $f=(f_j)_{j\in\mathbb{N}},g=(g_j)_{j\in\mathbb{N}}\in C(\mathbb{R})^\mathbb{N}$ and assume $\Phi_2(f)=\Phi_2(g)$. It follows that $f_l\ast\varphi_n=g_l\ast\varphi_n$ for all $(l,n)\in\mathbb{N}\times\mathbb{Z}$. This implies $$(f_l-g_l)\ast\sum_{n=-\infty}^{+\infty}\varphi_n=0,\quad\forall l\in\mathbb{N}.$$ Since $\sum_{n=-\infty}^{+\infty}\xi_n=1$ and $\overline{\mathcal{F}}(1)=\delta_0$ (i.e. the delta probability measure at the origin), we conclude $f_l-g_l=0$ for any $l\in\mathbb{N}$. Thus, $f=g$.
\end{proof}

\medskip

Next we indicate that $\left(\prod_{(l,n)\in\mathbb{N}\times\mathbb{Z}}\mathcal{B}^\mathbb{C}([(n-1)/2,(n+1)/2]),\mathbb{R},\sigma\right)$ can be embedded in $\left(\prod_{(l,n)\in\mathbb{N}\times\mathbb{Z}}\mathcal{B}([-(n+1)/2,-(n-1)/2]\cup[(n-1)/2,(n+1)/2])^2,\mathbb{R},\sigma\right)$ with the mapping $\Phi_3$ defined as follows: $$\prod_{(l,n)\in\mathbb{N}\times\mathbb{Z}}\mathcal{B}^\mathbb{C}\left(\left[\frac{n-1}{2},\frac{n+1}{2}\right]\right)\to\prod_{(l,n)\in\mathbb{N}\times\mathbb{Z}}\mathcal{B}\left(\left[-\frac{n+1}{2},-\frac{n-1}{2}\right]\bigcup\left[\frac{n-1}{2},\frac{n+1}{2}\right]\right)^2,$$$$\left(g_{(l,n)}\right)_{(l,n)\in\mathbb{N}\times\mathbb{Z}}\mapsto\left(\frac{g_{(l,n)}+\overline{g_{(l,n)}}}{2},\frac{g_{(l,n)}-\overline{g_{(l,n)}}}{2}\right)_{(l,n)\in\mathbb{N}\times\mathbb{Z}}.$$ Thus, this step (subsection) concludes by reenumerating the indices of the (latter) product space as well as the embedding mapping $\Phi_3$ affiliated.

\medskip

\subsection{Embedding $\left(\prod_{m=0}^{+\infty}\prod_{n=0}^m\mathcal{B}([-n/2-1,-n/2]\cup[n/2,n/2+1]),\mathbb{R},\sigma\right)$ in the universal real flow $\left(\prod_{m=0}^{+\infty}\prod_{n=0}^m\left(\mathcal{B}([-n/2-1,-n/2]\cup[n/2,n/2+1])\cap L(\mathbb{R})\right),\mathbb{R},\sigma\right)$}
This is the final step.

\medskip

Throughout this subsection we put $q_j=1/(j+1)$ for each $j\in\mathbb{N}$. In order to make all those functions attain to the Lipschitz constant $1$, we define for every $(l,j)\in\mathbb{N}\times\mathbb{N}$ a mapping $H_l^j:C(\mathbb{R})^\mathbb{N}\to L(\mathbb{R})$ by $$H_l^j(f)(t)=\frac12\int_t^{t+q_j}f_l(s)ds,\quad\forall f=(f_l)_{l\in\mathbb{N}}\in C(\mathbb{R})^\mathbb{N},\quad\forall t\in\mathbb{R}.$$ Note here that for any $(l,j)\in\mathbb{N}\times\mathbb{N}$ and any $f=(f_l)_{l\in\mathbb{N}}\in C(\mathbb{R})^\mathbb{N}$ we have indeed $|\!|H_l^j(f)|\!|_{L^\infty(\mathbb{R})}<1$ and $$\left|H_l^j(f)(s)-H_l^j(f)(t)\right|=\frac12\left|\int_s^tf_l(u)du-\int_{s+q_j}^{t+q_j}f_l(u)du\right|\le\left|s-t\right|,\quad\forall s,t\in\mathbb{R},$$ which implies $H_l^j(f)\in L(\mathbb{R})$.

We list three properties of the mapping $H_l^j$ (for an arbitrary $(l,j)\in\mathbb{N}\times\mathbb{N}$) as follows:
\begin{itemize}
\item $H_l^j:C(\mathbb{R})^\mathbb{N}\to L(\mathbb{R})$ is equivariant;
\item $H_l^j:C(\mathbb{R})^\mathbb{N}\to L(\mathbb{R})$ is continuous;
\item $\supp(\mathcal{F}(H_l^j(f)))\subset\supp(\mathcal{F}(f_l))$ for any $f=(f_l)_{l\in\mathbb{N}}\in C(\mathbb{R})^\mathbb{N}$.
\end{itemize}

\medskip

In fact, the first assertion follows from the equality $$H_l^j(\sigma_rf)(t)=\frac12\int_t^{t+q_j}f_l(s+r)ds=\frac12\int_{t+r}^{t+r+q_j}f_l(s)ds=\sigma_rH_l^j(f)(t)$$ for all $f=(f_l)_{l\in\mathbb{N}}\in C(\mathbb{R})^\mathbb{N}$ and all $t,r\in\mathbb{R}$.

To verify the second assertion, we fix a compact interval $[a,b]\subset\mathbb{R}$ arbitrarily. We take a sequence of functions $g^{(i)}=(g^{(i)}_l)_{l\in\mathbb{N}}\in C(\mathbb{R})^\mathbb{N}$ converging to $h=(h_l)_{l\in\mathbb{N}}\in C(\mathbb{R}^\mathbb{N})$ uniformly on the interval $[a,b+1]$ as $i\to+\infty$. This implies that $g^{(i)}_l\in C(\mathbb{R})$ converges to $h_l\in C(\mathbb{R})$ uniformly on $[a,b+1]$ as $i\to+\infty$. It follows that the sequence of functions $H_l^j(g^{(i)})\in L(\mathbb{R})$ converges to $H_l^j(h)\in L(\mathbb{R})$ uniformly on the interval $[a,b]$ as $i\to+\infty$. This shows the continuity of $H_l^j$.

To see the third property, it suffices to observe that $$2H_l^j(f)=\chi_{[-q_j,0]}\ast f_l,$$ where $\chi_{[-q_j,0]}:\mathbb{R}\to\{0,1\}$ is defined by $$\chi_{[-q_j,0]}(t)=\begin{cases}1,\quad&t\in[-q_j,0]\\0,\quad&t\notin[-q_j,0]\end{cases}.$$

\medskip

The above statements allow us to define a continuous and equivariant mapping $$\Phi_4:\prod_{m=0}^{+\infty}\prod_{n=0}^m\mathcal{B}\left(\left[-\frac{n}{2}-1,-\frac{n}{2}\right]\cup\left[\frac{n}{2},\frac{n}{2}+1\right]\right)$$$$\to\prod_{m=0}^{+\infty}\prod_{n=0}^m\left(\mathcal{B}\left(\left[-\frac{n}{2}-1,-\frac{n}{2}\right]\cup\left[\frac{n}{2},\frac{n}{2}+1\right]\right)\cap L\left(\mathbb{R}\right)\right),$$$$f=\left(f_l\right)_{l\in\mathbb{N}}\mapsto\left(H_l^j(f)\right)_{(l,j)\in\mathbb{N}\times\mathbb{N}}.$$

Note here that we have reenumerated the indices of the product spaces (both the former and the latter, respectively) in the definition of the mapping $\Phi_4$. More precisely, we need regard the former product space as a compact invariant subset of $C(\mathbb{R})^\mathbb{N}$, with a single index $l$ ranging over $\mathbb{N}$ instead, to which we apply the mapping $\Phi_4$ (as we wrote in the ``$\mapsto$'' line), and reenumerate the indices $(l,j)\in\mathbb{N}\times\mathbb{N}$ of the resulting product space, which is considered finally as the latter.

\medskip

To show that $\left(\prod_{m=0}^{+\infty}\prod_{n=0}^m\mathcal{B}([-n/2-1,-n/2]\cup[n/2,n/2+1]),\mathbb{R},\sigma\right)$ can be embedded in $\left(\prod_{m=0}^{+\infty}\prod_{n=0}^m\left(\mathcal{B}([-n/2-1,-n/2]\cup[n/2,n/2+1])\cap L(\mathbb{R})\right),\mathbb{R},\sigma\right)$ with the mapping $\Phi_4$, we have (by compactness) just its injectivity left over.

In fact, we take $g=(g_l)_{l\in\mathbb{N}},h=(h_l)_{l\in\mathbb{N}}\in C(\mathbb{R})^\mathbb{N}$ and assume $g\ne h$. Without loss of generality there exist $p\in\mathbb{N}$ and two real numbers $a<b$ such that $g_p(s)>h_p(s)$ for all $s\in[a,b]$. We choose $t\in\mathbb{R}$ and $j\in\mathbb{N}$ such that $a<t<t+q_j<b$. This results in $H_p^j(g)(t)>H_p^j(h)(t)$ which implies $\Phi_4(g)\ne\Phi_4(h)$. Therefore $\Phi_4$ is injective.

\medskip

So we have shown what we stated in the title of this subsection. Thus, the whole proof is eventually completed.

\medskip

\section{Further remarks}
This section aims to make a clearer picture of universal properties (for those function spaces mentioned in the previous sections) with a finite sequence of short remarks clarifying some potential deliberation.
\begin{remark}
We clarify that under the translation $\sigma$ the space $$\prod_{n=0}^{+\infty}\mathcal{B}([-n/3-1/2,-n/3]\cup[n/3,n/3+1/2])$$ is not universal. Furthermore, we may prove a variant of the Bebutov--Kakutani dynamical embedding theorem similar to Theorem \ref{thm:gjt} as follows: Let $0<\alpha<\beta$ be two real numbers. A real flow $(X,\mathbb{R},T)$ can be embedded in $$\left(\prod_{n=0}^{+\infty}\mathcal{B}([-n\alpha-\beta,-n\alpha]\cup[n\alpha,n\alpha+\beta]),\mathbb{R},\sigma\right)$$ if and only if the set of its fixed points $\{x\in X:T_tx=x,\,\forall t\in\mathbb{R}\}$ can be (topologically) embedded in $[0,1]$. We give a sketch of a proof: For the ``if'' part ``$\Longleftarrow$'' we employ Theorem \ref{thm:gjt} and follow the argument in the proof of the first main theorem, which allows us to embed $\left(L(\mathbb{R}),\mathbb{R},\sigma\right)$ in $\left(\prod_{n=0}^{+\infty}\mathcal{B}([-n\alpha-\beta,-n\alpha]\cup[n\alpha,n\alpha+\beta]),\mathbb{R},\sigma\right)$. For the ``only if'' part ``$\Longrightarrow$'' we note that for every positive integer $n$ the fixed-point set of the real flow $\left(\mathcal{B}([-n\alpha-\beta,-n\alpha]\cup[n\alpha,n\alpha+\beta]),\mathbb{R},\sigma\right)$ contains only one element, i.e. the constant function $0$, while for $n=0$ its fixed-point set consists of all constant functions $\mathbb{R}\to[-1,1]$, which is homeomorphic to $[0,1]$.
\end{remark}
\begin{remark}
We indicate that under the translation $\sigma$ the space $\mathcal{B}([-1,1])^\mathbb{N}$ is not universal. The outline of a proof is similar to our second main theorem. We take a (non-constant) bounded continuous function on $\mathbb{R}$ having a sufficiently small (positive) fundamental period, and consider its image in the space $\mathcal{B}([-1,1])^\mathbb{N}$, under an embedding mapping (if we assume the embeddability). Applying the sampling lemma (Lemma \ref{samplingtheorem}) we will deduce that only constant functions in $\mathcal{B}([-1,1])$ may attain such a small period (and thus, by equivariance, are the only possible candidates for the embedding image), which contradicts the injectivity of an embedding mapping. However, we notice that $\left(\mathcal{B}([-1,1])^\mathbb{N},\mathbb{Z},\sigma_1\right)$ is universal for $\mathbb{Z}$-actions. This fact follows from the interpolation lemma (Lemma \ref{interpolatingtheorem}).
\end{remark}
\begin{remark}
In this remark we prove the corollary: Under the translation $\sigma$ the space $\left(C^\infty(\mathbb{R})\cap L(\mathbb{R})\right)^\mathbb{N}$ is universal. This corollary follows from Theorem \ref{main}. In fact, a generalized Paley--Wiener theorem \cite[Lemma 2.2]{GT}\cite[Chapter 7, Section 8]{Schwartz} asserts that a bounded continuous function $f:\mathbb{R}\to\mathbb{C}$ satisfies $\supp(\mathcal{F}(f))\subset[-r,r]$ for some $r>0$ if and only if $f$ can be extended to a holomorphic function on $\mathbb{C}$ such that $|f(x+y\sqrt{-1})|\le e^{2\pi r|y|}\cdot|\!|f|\!|_{L^\infty(\mathbb{R})}$. Thus, for any positive real number $r$ we have in particular that all the functions in $\mathcal{B}([-r,r])$ must be analytic, and hence the space $\mathcal{B}([-r,r])\cap L(\mathbb{R})$ is a compact invariant subset of $C^\infty(\mathbb{R})\cap L(\mathbb{R})$.
\end{remark}
\begin{remark}
It is also possible to give a direct proof of the corollary. We describe a sketch of the proof as follows. As we mentioned in Section 5, it suffices to build a mapping $\Theta:C(\mathbb{R})^\mathbb{N}\to\left(C^\infty(\mathbb{R})\cap L(\mathbb{R})\right)^\mathbb{N}$ which is equivariant, continuous and injective. In order to make all the functions in $C(\mathbb{R})$ infinitely differentiable, we take a function $\theta$ on $\mathbb{R}$ as follows: $$\theta(t)=\begin{cases}c\cdot e^{1-t^2},&|t|<1\\0,&|t|\ge1\end{cases}$$ where the constant $c$ should be chosen to ensure $$\int_{-\infty}^{+\infty}\theta(t)\mathrm{d}t=1.$$ It is clear that $\theta$ is a nonnegative and smooth function on $\mathbb{R}$, supported in the compact interval $[-1,1]$, and it satisfies $\theta(t)=\theta(-t)$ for all $t\in\mathbb{R}$. For each $n\in\mathbb{N}$ we set $\theta_n(t)=n\theta(nt)$ for any $t\in\mathbb{R}$. Note that every $\theta_n$ has almost the same properties as $\theta$'s, but is supported in $[-n,n]$. Moreover, we can show that for any $f\in C(\mathbb{R})$ and $n\in\mathbb{N}$ the function $f\ast\theta_n$ is smooth, and for any $N\in\mathbb{N}$ we have $|\!|f\ast\theta_n-f|\!|_{L^\infty([-N,N])}\to0$ as $n\to\infty$. To conclude we apply the construction with the argument in Section 5 (twice) to the following diagram: $$C(\mathbb{R})^\mathbb{N}\longrightarrow C^\infty(\mathbb{R})^\mathbb{N}\longrightarrow\left(C^\infty(\mathbb{R})\cap L(\mathbb{R})\right)^\mathbb{N}.$$ More precisely, the former arrow corresponds to the embedding mapping $$(f_i)_{i\in\mathbb{N}}\mapsto\left(\left(f_i\ast\theta_j\right)_{i=1}^j\right)_{j=1}^{+\infty}$$ while the latter arrow corresponds to the embedding mapping which we constructed in (the final step of) Section 5.
\end{remark}
\begin{remark}
We would like to remark here that the following refinement of our main theorem is impossible. Let $C^\infty_c(\mathbb{R})$ be the set of all functions $f\in C^\infty(\mathbb{R})$ supported in a compact subset of $\mathbb{R}$. It is classically known that a continuous function can be written as a limit (in some sense) of a sequence of functions chosen in $C^\infty(\mathbb{R})$; and further, it is also feasible to require an approximation sequence of functions coming from $C^\infty_c(\mathbb{R})$. This fact leads naturally to a seemingly plausible question as follows:
\begin{itemize}
\item Is the space $\left(L(\mathbb{R})\cap C^\infty_c(\mathbb{R})\right)^\mathbb{N}$ universal under the translation $\sigma$?
\end{itemize}
However, it turns out that we cannot expect this space to be universal under the translation. In fact, such a space is (very) non-interesting for embedding. For example, if we choose a real flow possessing at least two distinct fixed points, then its (embedding) image must contain a nonzero constant function which thus does not have a compact support.
\end{remark}

\bigskip

\section*{Appendix}
The appendix is logically independent of the body of this paper. The only ingredient is to point out that the injectivity of an embedding mapping appearing in the Bebutov--Kakutani dynamical embedding theorem (in relation to Theorem \ref{thm:gjt} affiliated to Section 3) can be observed from finitely many points in the real line provided the phase space is finite dimensional and possesses no periodic points.

\medskip

In 1981 Takens established a well-known theorem in differential dynamical systems, a variant of which we may obtain within a Baire category framework essentially the same as the classical method in this direction.
\begin{itemize}\item
Suppose that $(X,\mathbb{R},T)$ is a real flow containing no periodic points. If $\dim(X)=d<+\infty$ then for every $(2d+1)$ distinct real numbers $r_0,r_1,\dots,r_{2d}$ there exists an embedding mapping $F:X\to C(\mathbb{R})$ which embeds $(X,\mathbb{R},T)$ in the translation on $C(\mathbb{R})$ and satisfies that for any two distinct points $x,x^\prime\in X$ there is some integer $0\le n\le2d$ such that $F(x)(r_n)\ne F(x^\prime)(r_n)$.
\end{itemize}

\medskip

This is an immediate consequence of the following statement:
\begin{itemize}\item
Suppose that $(X,\mathbb{R},T)$ is a real flow containing no periodic points. If $\dim(X)=d<+\infty$ then for every $(2d+1)$ distinct real numbers $r_0,r_1,\dots,r_{2d}$ there exists a continuous mapping $f:X\to[0,1]$ such that $$X\to[0,1]^{2d+1},\quad x\mapsto(f(T_{r_0}x),f(T_{r_1}x),\dots,f(T_{r_{2d}}x))$$ is a topological embedding mapping.
\end{itemize}

\medskip

We sketch the outline of its proof. For a more detailed and technical treatment we refer to \cite{Appendixcites1,Appendixcites2}. We denote $\Delta_X=\{(x,x):x\in X\}$. The strategy is to find for every pair $(x,x^\prime)\in X\times X\setminus\Delta_X$ an open neighbourhood $U_{(x,x^\prime)}\subset X\times X\setminus\Delta_X$ of $(x,x^\prime)$ satisfying that the set $$D_{U_{(x,x^\prime)}}^{r_0,\dots,r_{2d}}=\{f\in C(X,[0,1]):f_{r_0,\dots,r_{2d}}(y)\ne f_{r_0,\dots,r_{2d}}(y^\prime),\,\forall(y,y^\prime)\in\overline{U_{(x,x^\prime)}}\}$$ is open (which is easier and which we do not plan to explain here) and dense (which is harder and for which we will put an explanation in a moment) in the space $C(X,[0,1])$, where the continuous mapping $f_{r_0,\dots,r_{2d}}$ is defined by $$f_{r_0,\dots,r_{2d}}:X\to[0,1]^{2d+1},\quad x\mapsto(f(T_{r_0}x),f(T_{r_1}x),\dots,f(T_{r_{2d}}x)).$$ Since $\{U_{(x,x^\prime)}:(x,x^\prime)\in X\times X\setminus\Delta_X\}$ is an open cover of $X\times X\setminus\Delta_X$ which is a Lindel\"of space (namely, any of its open covers admits a countable subcover), there is a countable open cover $\{U_{(x_m,x^\prime_m)}:m\in\mathbb{N}\}$ of $X\times X\setminus\Delta_X$. By the Baire category theorem there exists a continuous mapping $f\in\bigcap_{m\in\mathbb{N}}D_{U_{(x_m,x^\prime_m)}}^{r_0,\dots,r_{2d}}$. Thus, the mapping $f_{r_0,\dots,r_{2d}}:X\to[0,1]^{2d+1}$ is injective.

To make $D_{U_{(x,x^\prime)}}^{r_0,\dots,r_{2d}}$ dense in $C(X,[0,1])$, let us suppose $x^\prime\ne T_tx$ for all $t\in\mathbb{R}$. We notice that the case $x^\prime=T_tx$ for some $t\in\mathbb{R}\setminus\{0\}$ is highly similar to what we assumed here, but should be with a more careful construction of perturbations. We take open neighbourhoods $U_x$ and $U_{x^\prime}$ of $x$ and $x^\prime$, respectively, such that the sets $T_{r_0}(\overline{U_x}),\dots,T_{r_{2d}}(\overline{U_x})$, $T_{r_0}(\overline{U_{x^\prime}}),\dots,T_{r_{2d}}(\overline{U_{x^\prime}})$ are pairwise disjoint, and set $U_{(x,x^\prime)}=U_x\times U_{x^\prime}$. So we have already defined $D_{U_{(x,x^\prime)}}^{r_0,\dots,r_{2d}}$. Now we fix $f\in C(X,[0,1])$ and $\epsilon>0$ arbitrarily. Noting that $\dim(X)=d<+\infty$ we need choose finite open covers $\alpha_x$ and $\alpha_{x^\prime}$ of $\overline{U_x}$ and $\overline{U_{x^\prime}}$, respectively, which are sufficiently fine, such that $\mathrm{diam}(f(T_{r_n}(V)))<\epsilon/2$ for each $V\in\alpha_x\cup\alpha_{x^\prime}$ and each integer $0\le n\le2d$. Let $w\in\{x,x^\prime\}$. We take a partition of unity $\{\psi_V^w\}_{V\in\alpha_w}$ of $\overline{U_w}$ subordinate to $\alpha_w$, namely, a family of continuous functions $\psi_V^w:\overline{U_w}\to[0,1]$ satisfying: $$\sum_{V\in\alpha_w}\psi_V^w(z)=1,\quad\forall\,z\in\overline{U_w};\quad\quad\mathrm{supp}(\psi_V^w)\subset V,\quad\forall\,V\in\alpha_w.$$ Without loss of generality we can choose pairwise distinct points $p_V^w\in V$, for every $V\in\alpha_w$ and for $w\in\{x,x^\prime\}$, with $\psi_V^w(p_V^w)=1$. We set for any $V\in\alpha_w$ a vector $$u_V^w=\left(f(T_{r_0}p_V^w),\dots,f(T_{r_{2d}}p_V^w)\right)\;\in[0,1]^{2d+1}.$$ We need find $q_V^w\in[0,1]^{2d+1}$, for each $V\in\alpha_w$, with $|\!|q_V^w-u_V^w|\!|_\infty<\epsilon/2$, such that any $(2d+2)$ pairwise distinct vectors in the family $\{q_V^w:V\in\alpha_w,w\in\{x,x^\prime\}\}$ are affinely independent in $\mathbb{R}^{2d+1}$. Next we define a continuous mapping $$k_w:\overline{U_w}\to[0,1]^{2d+1},\quad z\mapsto\sum_{V\in\alpha_w}\psi_V^w(z)q_V^w.$$ We put $W=\bigcup_{n=0}^{2d}\bigcup_{w\in\{x,x^\prime\}}T_{r_n}(\overline{U_w})$ and let $g:W\to[0,1]$ be a continuous function defined by $g(T_{r_n}z)=\mathsf{proj}_n(k_w)(z)$ for any integer $0\le n\le2d$ and any $z\in\overline{U_w}$, where $w$ ranges over $\{x,x^\prime\}$. We may verify $|\!|g-f|_W|\!|_\infty<\epsilon$. Finally it suffices to extend $g:W\to[0,1]$ to a continuous function $h:X\to[0,1]$ (i.e. $h|_W=g$) with $|\!|f-h|\!|_\infty<\epsilon$, and to show $h\in D_{U_{(x,x^\prime)}}^{r_0,\dots,r_{2d}}$, as required.

\medskip

\end{document}